\documentclass{article}%
\usepackage{amssymb}
\usepackage{amsfonts}
\usepackage{graphicx}
\usepackage{amsmath}%
\providecommand{\U}[1]{\protect\rule{.1in}{.1in}}
\newtheorem{theorem}{Theorem}

\newtheorem{corollary}[theorem]{Corollary}

\newtheorem{definition}[theorem]{Definition}

\newtheorem{lemma}[theorem]{Lemma}

\newtheorem{remark}[theorem]{Remark}

\newenvironment{proof}[1][Proof]{\textbf{#1.} }{\ \rule{0.5em}{0.5em}}
\begin{document}

\begin{center}
{\LARGE Outer actions of finite groups on prime C*-algebras and duality}

\bigskip

\bigskip

Costel Peligrad
\end{center}

\bigskip

Department of Mathematical Sciences, University of Cincinnati, PO Box 210025,
Cincinnati, OH 45221-0025, USA. E-mail address: costel.peligrad@uc.edu

Key words and phrases. dynamical system, algebra of local multipliers,
strictly outer actions.

2020 Mathematics Subject Classification. Primary 46L55, 46L40; Secondary 20F29.

\bigskip

\textbf{ABSTRACT. An action of a compact, in particular finite group on a
C*-algebra is called properly outer if no automorphism of the group that is
distinct from identity is implemented by a unitary element of the algebra of
local multipliers of the C*-algebra and strictly outer if the commutant of the
algebra in the algebra of local mutipliers of the cross product consists of
scalars [11]. In [11, Theorem 11] I proved that for finite groups and prime
C*-algebras (not necessarily separable), the two notions are equivalent. I
also proved that for finite abelian groups this is equivalent to other
relevant properties of the action [11 Theorem 14]. In this paper I add other
properties to the list in [11, Theorem 14]. }

\textbf{ }

\bigskip In this paper, the C*-algebras under consideration are not
necessarily separable. Let $B$\ be a C*-algebra. Then the multiplier algebra
$M(B)$\ of $B$\ is defined as $M(B)=\left\{  b\in B^{\prime\prime}:ba\in
B\text{ and }ab\in B\text{ for all }a\in B\right\}  $ where $B^{\prime\prime}%
$\ is the bidual of $B$ [10]. Notice that if $\left\{  e_{\lambda}\right\}  $
is an approximate unit of $B$\ then the strong limit of $\left\{  e_{\lambda
}\right\}  $ in $B^{\prime\prime}$\ is the unit of $B^{\prime\prime}.$
Therefore $M(B)$\ is a unital C*-algebra that contains $B$\ as an essential
ideal (a two sided ideal $I\subset B$\ is called essential if $\left\{  x\in
B:xa=0\text{ for all }a\in B\right\}  =\left\{  0\right\}  .$ \ If $A$\ is a
C*-algebra, denote by $\mathcal{I}$\ the set of all closed essential, two
sided ideals of $A$. If $I\subset J$\ are two such ideals then according to
[10, Proposition 3.12.8.] we have the inclusion $M(J)\subset M(I)$\ where
$M(I),$\ respectively $M(J)$ are the multiplier algebras of $I,$\ respectively
$J.$\ The inductive limit of the system $\left\{  M(I):I\in\mathcal{I}%
\right\}  $\ is called the algebra of local multipliers of $A$\ and denoted
$M_{loc}(A)$ [1]$.$ Originally, this algebra was denoted $M^{\infty}(A)$\ and
called the algebra of essential multipliers of $A$\ [9]. Recall that a
C*-algebra, $A,$ is called prime if every (two sided) ideal of $A$\ is essential.

In [2 Proposition 3.2.] the authors prove that three possible definitions of
properly outer automorphisms of separable C*-algebras are equivalent. In [11]
I have chosen one of those definititions for not necessarily separable
C*-algebras, namely:

\begin{definition}
An automorphism of a (not necessarily separable) C*-algebra is called properly
outer if it is not implemented by a unitary element of $M_{loc}(A).$
\end{definition}

This definition implies the other two equivalent statements of [2, Proposition
3.2.] for not necessarily separable C*-algebras.

In the rest of this paper we will consider a C*-dynamical system
$(A,G,\alpha)$\ where $A$\ is a C*-algebra, $G$\ a finite group and $\alpha
$\ a faithful action of $G$\ on $A$ (the action $\alpha$\ is called faithful
if $\alpha_{g}\neq id$ for $g\neq e$\ where $e$\ is the unit of $G$ and
$id\ $is the identity automorphism of $A).$ If the group $G$ is finite, the
action $\alpha$ of $G$ on\ $A$\ has a natural extension to $M_{loc}(A)$ (this
fact is also true for compact groups and prime C*-algebras that have a
faithful factorial representation [12], in particular separable prime
C*-algebras)$.$

\bigskip The cross product $A\times_{\alpha}G$ of $A$\ by a finite group,
$G,$\ is defined as the C*-completion of the algebra of all functions
$f:G\rightarrow A$ with the following operations:%
\[
(f\cdot h)(g)=\sum_{p\in G}f(p)\alpha_{p}(h(p^{-1}g))
\]%
\[
f^{\ast}(g)=\alpha_{g}(f(g^{-1})^{\ast})
\]
As $G$\ is a finite group, each $f\in A\times_{\alpha}G$\ can be written
$f=\sum a_{g}\delta_{g}$\ where $a_{g}\in A$\ and $\delta_{g}$\ is the
function on $G$ with values in the multiplier algebra $M(A)$\ of $A$\ such
that $\delta_{g}(g)=1$\ and $\delta_{g}(p)=0$\ for $p\neq g.$

In [11], I defined the notion of strict outerness of the action $\alpha$ of a
finite group $G$\ on a C*-algebra $A$\ (similar to the corresponding
definition for actions on von Neumann factors in [13]) and proved that it is
equivalent with proper outerness [11, Theorem 11] :

\begin{definition}
T\bigskip he action $\alpha$ is called strictly outer if $A^{\prime}\cap
M_{loc}(A\times_{\alpha}G)=%
\mathbb{C}
I,$\ where $A^{\prime}$\ is the commutant of $A.$
\end{definition}

\bigskip

In Theorem 12 below I will discuss several conditions on a C*-dynamical system
$(A,G,\alpha)$ with\ $A$\ prime and $G$\ finite abelian that are equivalent
with the proper outerness and the strict outerness of $\alpha.$

Let $G$\ be a compact (in particular finite) abelian group. If $\widehat{G}%
$\ denotes the dual of $G$\ and $\gamma\in\widehat{G},$\ let
\[
A_{\gamma}=\left\{  a\in A:\alpha_{g}(a)=\left\langle g,\gamma\right\rangle
a,g\in G\right\}  .
\]
If $\gamma=e$ the unit of $\widehat{G}$\ we denote $A_{e}=A^{\alpha}$ and call
it the fixed point algebra of the action $\alpha.$ Then the mapping
\[
P^{\alpha}(a)=\int\alpha_{g}(a)dg.
\]
is a faithful norm one conditional expectation $P^{\alpha}:A\rightarrow
A^{\alpha}$. If $A$\ has a factorial representation (in particular if $A$\ is
separable), then, $P^{\alpha}$\ admits a natural extension to a faithful norm
one conditional expectation $P^{\alpha}:M_{loc}(A)\rightarrow M_{loc}%
(A)^{\alpha}$ even for non abelian groups ([12, Lemma 1.2.]). If $G$\ is
finite (not necessarily abelian), then, since the restriction of an
automorphism to an invariant ideal has an extension to its multiplier algebra,
it is obvious that for every C*-algebra, $A,$ $P^{\alpha}$\ admits a natural
extension to a faithful norm one conditional expectation of $M_{loc}(A)$ onto
$M_{loc}(A)^{\alpha}.$ In this case (of finite groups)%
\[
P^{\alpha}(a)=\frac{1}{\left\vert G\right\vert }\sum_{g\in G}\alpha_{g}(a)
\]
where $\left\vert G\right\vert $\ is the cardinality of $G.$%
\[
\]

We will denote by $sp(\alpha)$\ the Arveson spectrum of $\alpha$\ and by
$\ \widetilde{sp}\ (\alpha)$\ the strong Arveson spectrum of $\alpha$:
\[
sp(\alpha)=\left\{  \gamma\in\widehat{G}:A_{\gamma}\neq\left\{  0\right\}
\right\}  .
\]%
\[
\ \widetilde{sp}\ (\alpha)=\left\{  \gamma\in\widehat{G}:\overline{A_{\gamma
}A_{\gamma}^{\ast}}=A^{\alpha}\right\}  .
\]
and, if $\mathcal{H}$\ is the set of all non zero $\alpha$-invariant
hereditary subalgebras of $A,$ define the Connes spectrum $\Gamma(\alpha)$ and
the strong Connes spectrum $\widetilde{\Gamma}(\alpha)$ of Kishimoto
\[
\Gamma(\alpha)=\cap\left\{  sp(\alpha|_{C}:C\in\mathcal{H}\right\}  .
\]%
\[
\widetilde{\Gamma}(\alpha)=\cap\left\{  \widetilde{sp}(\alpha|_{C}%
:C\in\mathcal{H}\right\}  .
\]
If $\mathcal{H}_{B}$\ denotes the set of $\alpha$-invariant hereditary
C*-subalgebras of $A$, $C,$\ such that the two sided ideal generated by
$C$\ is essential in $A$, the Borchers spectrum of $\alpha$\ is by definition%
\[
\Gamma_{B}(\alpha)=\cap\left\{  sp(\alpha)|_{C}:C\in\mathcal{H}_{B}\right\}
.
\]
It is obvious that, if $A$\ is prime then $\Gamma_{B}(\alpha)=\Gamma(\alpha).$

\bigskip The next results and definitions which are valid for general, not
necessarily separable C*-algebras, will be used in the proof of the main
result of this paper (Theorem 12).

\bigskip

\begin{lemma}
a) [11, Lemma 3.a)] If $(A,G,\alpha)$ is a C*-dynamical system with
$G$\ finite then $M_{loc}(A)^{\alpha}\subset M_{loc}(A^{\alpha})$\newline b)
[11, Lemma 3.d)] If $A$\ is prime then $A^{\prime}\cap M_{loc}(A)$ consists of scalars.
\end{lemma}

\bigskip

\begin{lemma}
(5, Lemma 3.2.) \bigskip Let a be a positive element of the C* -algebra A,
$\left\{  a_{i}:i=1,...n\right\}  $elements of $A$, $\left\{  \alpha
_{i}:i=1,...n\right\}  $ automorphisms of $A$ with $\widetilde{\Gamma}%
(\alpha_{i})\neq\left\{  1\right\}  $ and $\epsilon>0.$ Then, there exists a
positive $x\in A$ with $\left\Vert x\right\Vert =1$ such that%
\[
\left\Vert xax\right\Vert \geqslant\left\Vert a\right\Vert -\epsilon\text{ and
}\left\Vert xa_{i}\alpha_{i}(x)\right\Vert \leq\epsilon
\]
\ 
\end{lemma}

\bigskip

The key fact in the proof of the above Lemma is Lemma 1.1. in that paper. In a
subsequent paper, [6, Theorem 2.1.], Kishimoto has replaced the strong Connes
spectrum in [5, Lemma 1.1.] with the Borchers spectrum, $\Gamma_{B}(\alpha
_{i}|_{J}),$where $J$\ is an arbitrary $\alpha_{i}$-invariant ideal. If
$A$\ is prime, this condition is equivalent with the condition $\Gamma
(\alpha_{i})\neq\left\{  1\right\}  .$ Therefore

\bigskip

\begin{corollary}
Suppose $A$\ is a prime C*-algebra, $a$ a positive element of $A$, $\left\{
a_{i}:i=1,...n\right\}  $ elements of $A$, $\left\{  \alpha_{i}%
:i=1,...n\right\}  $ automorphisms of $A$ with $\Gamma(\alpha_{i})\neq\left\{
1\right\}  $ and $\epsilon>0.$ Then, there exists a positive $x\in A$ with
$\left\Vert x\right\Vert =1$ such that%
\[
\left\Vert xax\right\Vert \geqslant\left\Vert a\right\Vert -\epsilon\text{ and
}\left\Vert xa_{i}\alpha_{i}(x)\right\Vert \leq\epsilon
\]

\end{corollary}

\bigskip

If $(A,G,\alpha)$\ is a dynamical system with \ $A$ prime and $G$\ finite,
then, every (two sided) ideal $I\subset A$ contains an $\alpha$-invariant
ideal, namely $\cap\left\{  \alpha_{g}(I):g\in G\right\}  .$ Therefore,
$M_{loc}(A)$\ is the inductive limit of $\left\{  M(I):I\text{ }%
\alpha-\text{invariant ideal of }A\right\}  .$

In the next result I will show that if $\alpha$\ is an automorphism of a prime
C*-algebra that generates a finite group of automorphisms (i.e. $\alpha
^{n}=id$\ for some $n\in%
\mathbb{N}
)$, then, it is properly outer in the sense of our Definition 1 if and only if
$\Gamma(\alpha)\neq\left\{  1\right\}  .$ Part of the proof uses some ideas
from [10, Theorem 8.9.7.] and [4, Proposition 7].

\bigskip

\begin{lemma}
Let $\alpha$\ be an automorphism of a prime C*-algebra $A$\ such that
$\alpha^{n}=id$\ for some $n\in%
\mathbb{N}
.$ Then, $\alpha$\ is properly outer if and only if $\ \Gamma(\alpha
)\neq\left\{  1\right\}  .$
\end{lemma}

\begin{proof}
\bigskip Suppose that $\alpha$\ is not properly outer, so there exists an
unitary element $u_{0}\in M_{loc}(A)$ such that $\alpha=Adu_{0}.$ Since
$\alpha$\ has a natural extension to $M_{loc}(A)$ and $\alpha(u_{0}%
)=u_{0}u_{0}u_{0}^{\ast}=u_{0}$, it follows that $u_{0}\in M_{loc}(A)^{\alpha
}.$ Let $G=%
\mathbb{Z}
_{m}$\ where $m=\min\left\{  n\in%
\mathbb{N}
:\alpha^{n}=id\right\}  $ and $\widehat{G}$ its dual group. Denote by $G_{0}$
the annihilator of $\Gamma(\alpha)\subset\widehat{G}$\ in $G.$ I will prove
that $G_{0}=G$, so $\Gamma(\alpha)=\left\{  1\right\}  .$ Since $G=%
\mathbb{Z}
_{m},$\ it is enough to prove that $1\in G_{0}$. Since, as noticed above
$M_{loc}(A)$\ is the inductive limit of $\left\{  M(I\right\}  :I$ $\alpha
-$invariant ideal of $A\},$ it follows that $M_{loc}(A)^{\alpha}$ is the
inductive limit of $\left\{  M(I\right\}  ^{\alpha}:I$ $\alpha-$invariant
ideal of $A\}.$ Suppose that $G_{0}\varsubsetneq G,$\ so $\Gamma(\alpha
)\neq\left\{  1\right\}  .$ Since $G$\ is finite, $\Gamma(\alpha)=\left\{
1,\gamma_{1},...\gamma_{k}\right\}  $\ is also finite. Let $d=\min\left\{
\left\vert \left\langle 1,\gamma_{i}\right\rangle -1\right\vert
:i=1,...k\right\}  $ and $\epsilon<d.$ Since $M_{loc}(A)^{\alpha}$ is the
inductive limit of $\left\{  M(I\right\}  ^{\alpha}:I$ $\alpha-$invariant
ideal of $A\},$ by [14, Proposition L.2. 2.] there exists an $\alpha
$-invariant ideal $I\subset A$\ and a unitary $u\in M(I)^{\alpha}$\ such that
$\left\Vert u_{0}-u\right\Vert <\frac{\epsilon}{4}.$ Choose $\lambda_{0}\in
Sp(u)$ and $f$\ a positive continuous function on the unit circle
$\boldsymbol{T}$\ with $suppf\subset\left\{  \lambda\in\boldsymbol{T}%
:\left\vert \lambda-\lambda_{0}\right\vert <\frac{\epsilon}{4}\right\}  .$
Then $f(u)\in M(I)^{\alpha}$. Let $C$\ be the ($\alpha$-invariant) hereditary
C*-subalgebra of $I$ generated by $f(u)If(u).$ Since $I$\ is an ideal, so
hereditary subalgebra of $A$,\ it follows that $C$\ is an $\alpha$-invariant
C*-subalgebra of $A.$\ Let $x\in C.$ Then%
\begin{align*}
\left\Vert \alpha(x)-x\right\Vert  &  =\left\Vert u_{0}x-xu_{0}\right\Vert
\leqslant\left\Vert (u_{0}-u)x-x(u_{0}-u)\right\Vert +\left\Vert
ux-xu\right\Vert <\\
&  <\frac{\epsilon}{2}\left\Vert x\right\Vert +\left\Vert ux-xu\right\Vert .
\end{align*}
Since $x\in C$\ we have%
\[
\left\Vert ux-xu\right\Vert =\left\Vert (u-\lambda_{0})x-x(u-\lambda
_{0})\right\Vert <\frac{\epsilon}{2}\left\Vert x\right\Vert .
\]
Therefore,
\[
\left\Vert \alpha(x)-x\right\Vert <\epsilon\left\Vert x\right\Vert .
\]
Let $\gamma\in\Gamma(\alpha),\gamma\neq1.$ Since $\gamma\in Sp(\alpha|_{C}%
)$\ there exists $x\in C,x\neq0$ such that $\alpha(x)=\left\langle
1,\gamma\right\rangle x.$ Therefore%
\[
\left\Vert \left\langle 1,\gamma\right\rangle x-x\right\Vert =\left\vert
\left\langle 1,\gamma\right\rangle -1\right\vert \left\Vert x\right\Vert
<\epsilon\left\Vert x\right\Vert
\]
So\ $\left\vert \left\langle 1,\gamma\right\rangle -1\right\vert <\epsilon<d.$
This contradiction shows that $\Gamma(\alpha)=\left\{  1\right\}  .$

Conversely, suppose that $\Gamma(\alpha)=\left\{  1\right\}  .$ Then
$G_{0}=G.$ In particular $1\in G_{0},$ that is $\left\langle 1,\gamma
\right\rangle =1$ for $\gamma\in\Gamma(\alpha)$ $(=\left\{  1\right\}  ).$ By
[11, Proposition 7.], there exists an $\alpha$-invariant essential ideal
$J$\ and a unitary $u\in M(J)^{\alpha}$ \ such that $\alpha=Adu.$ Therefore,
since $M(J)^{\alpha}\subset M_{loc}(A)^{\alpha}\subset M_{loc}(A),$ $\alpha$
\ is not properly outer.\ 
\end{proof}

\bigskip

For separable C*-algebras the hypothesis in the above Lemma that the
automorphism generates a finite group is not necessary ([8, Theorem 6.6.]). I
do not know if this is true for non separable C*-algebras.

\bigskip The following Corollary is a reformulation of Corollary 5 using the
above Lemma 6

\begin{corollary}
Let $(B,H,\beta)$\ be a C*-dynamical system with $B$\ prime and $H$ finite, a
a positive element of $B$, $\left\{  a_{i}:i=1,...n\right\}  $ elements of
$B$, $\left\{  \beta_{i}:i=1,...n\right\}  \subset H$ a subset of properly
outer automorphisms of $B$ and $\epsilon>0.$ Then, there exists a positive
$x\in B$ with $\left\Vert x\right\Vert =1$ such that%
\[
\left\Vert xax\right\Vert \geqslant\left\Vert a\right\Vert -\epsilon\text{ and
}\left\Vert xa_{i}\beta_{i}(x)\right\Vert \leq\epsilon
\]

\end{corollary}

\bigskip I will discuss next the notion of topological transitivity for groups
of automorphisms of a C*-algebra and some duality results for automorphisms.
If $X$ is a locally compact topological space, a homeomorphism $\varphi$\ of
$X$\ is called topologically transitive if for every pair of non empty
$\varphi$-invariant open subsets $U_{1},U_{2}$ their intersection $U_{1}\cap
U_{2}$ is non empty. For general C*-algebras we defined

\begin{definition}
\bigskip\lbrack7, Definition 2.1. b)] Let $(A,H,\beta)$ be a C*-dynamical
system. The action $\beta$\ is called topologically transitive if for every
pair of non zero $\beta$-invariant hereditary C*-subalgebras $B_{1},B_{2}$\ of
$A,$\ their product $B_{1}B_{2}$\ is non zero.
\end{definition}

\bigskip The following form of the above definition has been noticed in [3]:

\begin{remark}
\bigskip\lbrack3]The action $\beta$\ is topologically transitive if and only
if for every pair of non zero elements $x,y\in$ $A$ \ there exists $h\in H$
such that $x\beta_{h}(y)\neq0.$
\end{remark}

In [7, Theorem 3.1.] we have proved the following duality result for
automorphisms of C*-algebras:

\begin{theorem}
Let $(A,G,\alpha)$\ be a system with $G$\ compact abelian. Suppose that there
exists a topologically transitive group of automorphisms $(H,\beta)$ of
$A$\ that commutes with $\alpha,$ that is $\alpha_{g}\beta_{h}=\beta_{h}%
\alpha_{g}$ for all $g\in G,h\in H.$ If $\sigma$\ is an automorphism of
$A$\ that leaves $A^{\alpha}$\ pointwise invariant and commutes with $\beta
$\ then there exists $g\in G$\ such that $\sigma=\alpha_{g}.$
\end{theorem}

In [11] we have proved the following

\begin{theorem}
\bigskip Let $(A,G,\alpha)$\ be a C*-dynamical system with $A$\ prime,
$G$\ finite abelian and $\alpha$\ faithful. The following conditions are
equivalent\newline a) $\alpha_{g}$ is properly outer for every $g\in
G,g\neq0.$\newline b) $(A^{\alpha})^{\prime}\cap M_{loc}(A)=%
\mathbb{C}
I.$\newline c) $(A)^{\prime}\cap M_{loc}(A\times_{\alpha}G)=%
\mathbb{C}
I.$\newline d) $A^{\alpha}$\ is prime and $\widehat{\alpha}_{\gamma}$\ is
properly outer (in $M_{loc}(A\times_{\alpha}G))$ for every $\gamma
\in\widehat{G},\gamma\neq0$\newline
\end{theorem}

If $(A,G,\alpha)$\ is a C* dynamical system with $A$\ prime then, since $A$ is
an essential ideal of $M(A),$\ from [9, page 302 statement 3)], we have
$M_{loc}(M(A))=M_{loc}(A).$ It follows that if $\alpha_{g}$\ is a properly
outer automorphism of $A$\ then the extension of $\alpha_{g}$\ to $M(A)$\ is
also properly outer. Therefore the above Theorem 11 can be stated for
$M(A)$\ instead of $A.$

\begin{theorem}
Let $(A,G,\alpha)$\ be a C*-dynamical system with $A$\ prime, $G$\ finite
abelian and $\alpha$\ faithful. Then the statements a)-d) in Theorem 11 are
equivalent to each of the following\newline\newline e) If $x,y\in M_{loc}(A),$
are nonzero elements, then $xM_{loc}(A)^{\alpha}y\neq\left\{  0\right\}
.$\newline f) If \ $\mathcal{U}$\ denotes the group of unitary elements of
$M_{loc}(A)^{\alpha}$ then the group $H=\left\{  Adu:u\in\mathcal{U}\right\}
$\ acts topologically transitively on $M_{loc}(A).$\newline g) $A^{\alpha}$ is
prime and if $\beta$\ is an automorphism of $M_{loc}(A)$\ that leaves
$A^{\alpha}$\ pointwise invariant then there exists $g\in G$\ such that
$\beta=\alpha_{g}.$
\end{theorem}

\begin{proof}
d)$\Rightarrow$e) It is sufficient to prove e) if $x,y$\ are nonzero positive
elements of $M_{loc}(A).$\ In [2, 4$\Rightarrow$1] the authors proved the
following inequality for compact groups and separable C*-algebras
\[
\sup\left\{  \left\Vert xay\right\Vert :a\in A^{\alpha},\left\Vert
a\right\Vert \leqslant1\right\}  \geqslant\left\Vert P^{\alpha}(x)\right\Vert
\left\Vert P^{\alpha}(y)\right\Vert .\text{ \ (1)}%
\]
for every pair of nonzero positive elements $x,y\in A.$ Their proof is valid
for not necessarily separable prime C*-algebras $A$\ and finite groups $G$\ if
instead of using [8, Lemma 7.1.] or [5, Lemma 3.2.], as the authors do, one
uses Corollary 7 in the current paper. Also, according to the discussion after
Theorem 11, we have for every pair of nonzero positive elements $x,y\in M(A)$%
\[
\sup\left\{  \left\Vert xay\right\Vert :a\in M(A)^{\alpha},\left\Vert
a\right\Vert \leqslant1\right\}  \geqslant\left\Vert P^{\alpha}(x)\right\Vert
\left\Vert P^{\alpha}(y)\right\Vert .\text{ \ (2)}%
\]

Now let $x,y\in M_{loc}(A)$\ be nonzero positive elements of norm $1$, so
since $P^{\alpha}\ $is a faithful conditional expectation, $P^{\alpha}(x)$ and
$P^{\alpha}(y)$ are nonzero positive elements$.$ Let $d=\min\left\{
\left\Vert P^{\alpha}(x)\right\Vert ,\left\Vert P^{\alpha}(y)\right\Vert
\right\}  .$ Then according to \ [14, Proposition L. 2. 2.], for every $n\in%
\mathbb{N}
,$ in particular for $n>\frac{2}{d}+2$,\ there exists an $\alpha$-invariant
ideal $I_{n}$ of $A$\ and $x_{n},y_{n}$\ nonzero positive elements of
$M(I_{n})$\ such that $\left\Vert x_{n}-x\right\Vert <\frac{d}{n}$\ and
$\left\Vert y_{n}-y\right\Vert <\frac{d}{n}$, hence $\left\Vert P^{\alpha
}(x_{n})-P^{\alpha}(x)\right\Vert <\frac{d}{n}$ and $\left\Vert P^{\alpha
}(y_{n})-P^{\alpha}(y)\right\Vert <\frac{d}{n}$, so
\[
\left\Vert P^{\alpha}(x_{n})\right\Vert >\left\Vert P^{\alpha}(x)\right\Vert
-\frac{d}{n}\ \text{and }\left\Vert P^{\alpha}(y_{n})\right\Vert >\left\Vert
P^{\alpha}(y)\right\Vert -\frac{d}{n}.\text{ \ \ (3)}%
\]
Also, since $(norm)\lim x_{n}=x$, $(norm)\lim y_{n}=y$ and $x,y$ are norm
$1$\ elements, we can take $x_{n},y_{n}$\ of norm $1.$\ Since by the
inequality (2) applied to $M(I_{n})$\ instead of $M(A)$\ we have%
\[
\sup\left\{  \left\Vert x_{n}ay_{n}\right\Vert :a\in M(I_{n})^{\alpha
},\left\Vert a\right\Vert \leqslant1\right\}  \geqslant\left\Vert P^{\alpha
}(x_{n})\right\Vert \left\Vert P^{\alpha}(y_{n})\right\Vert
\]
there exists $a_{n}\in M(I_{n})^{\alpha},\left\Vert a_{n}\right\Vert
\leqslant1$ such that
\[
\left\Vert x_{n}a_{n}y_{n}\right\Vert >\left\Vert P^{\alpha}(x_{n})\right\Vert
\left\Vert P^{\alpha}(y_{n})\right\Vert -\frac{d^{2}}{n^{2}}.\text{
\ \ \ \ \ \ \ \ \ \ \ \ \ \ \ \ \ \ \ \ \ \ \ \ \ \ \ \ \ \ \ \ \ \ (4)}%
\]
Next, notice that
\[
xa_{n}y=x_{n}a_{n}y_{n}-(x_{n}-x)a_{n}y_{n}-xa_{n}(y_{n}%
-y).\text{\ \ \ \ \ \ \ \ \ \ \ \ \ \ \ \ \ \ \ \ }%
\]
Therefore%
\begin{align*}
\left\Vert xa_{n}y\right\Vert  &  \geqslant\left\Vert x_{n}a_{n}%
y_{n}\right\Vert -\left\Vert (x_{n}-x)a_{n}y_{n}\right\Vert -\left\Vert
xa_{n}(y_{n}-y)\right\Vert >\\
&  \left\Vert x_{n}a_{n}y_{n}\right\Vert -\frac{2d}{n}%
\text{.\ \ \ \ \ \ \ \ \ \ \ \ \ \ \ \ \ \ \ \ \ \ \ \ \ \ \ \ \ \ \ \ \ \ \ \ \ \ \ \ \ \ \ \ \ \ \ \ \ \ \ \ \ \ \ \ (5)}%
\end{align*}
From (3), (4) and (5) it follows that%
\begin{align*}
\left\Vert xa_{n}y\right\Vert  &  >(\left\Vert P^{\alpha}(x)\right\Vert
-\frac{d}{n})(\left\Vert P^{\alpha}(y)\right\Vert -\frac{d}{n})-\frac{2d}%
{n}-\frac{d^{2}}{n^{2}}\geqslant d^{2}-\frac{2d}{n}-\frac{2d^{2}}{n}>0.\\
&  .
\end{align*}
so $xM_{loc}(A)^{\alpha}y\neq\left\{  0\right\}  .$\ \ \ \ \ \ \ \ \ \ \ \ \ \ \ \ \ \ \ \ \ \ \ \ \ \ \ \ \ \ \ \ \ \ \ \ \ \ \ \ \ \ \ \ \ \ \ \ \ \ \ \ \ \ 

e)$\Leftrightarrow$f)\ \ Let $x,y$ be nonzero elements of $M_{loc}(A).$\ Since
$xM_{loc}(A)^{\alpha}y\neq\left\{  0\right\}  ,$ then, in particular,
$x\mathcal{U}y\neq\left\{  0\right\}  ,$ so\ there exists $u\in\mathcal{U}%
$\ such that $xAdu(y)\neq0.$ By Remark 9 it follows that $H=\left\{
Adu:u\in\mathcal{U}\right\}  $\ acts topologically transitively on
$M_{loc}(A).$ The converse follows from the fact that the linear\ span of
$\mathcal{U}$\ equals $M_{loc}(A)^{\alpha}.$

f)$\Rightarrow$g) Let $x,y$\ be nonzero elements of $A^{\alpha}$.\ Then, by
e), $xM_{loc}(A)^{\alpha}y\neq\left\{  0\right\}  .$ Therefore, there exists
an $\alpha$-invariant ideal, $I\subset A$ such that $xM(I)^{\alpha}%
y\neq\left\{  0\right\}  .$ Since $M(I)^{\alpha}=M(I^{\alpha})$ [11, Lemma 2
b)], it follows that $xI^{\alpha}y\neq\left\{  0\right\}  ,$ so $xA^{\alpha
}y\neq\left\{  0\right\}  $ and thus $A^{\alpha}$\ is prime. Now, let $\beta
$\ be an automorphism of $M_{loc}(A)$ that leaves $A^{\alpha}$\ pointwise
invariant. Then for every $\alpha$-invariant ideal $I\subset A,\beta$ leaves
$I^{\alpha}$\ and, therefore $M(I^{\alpha})=M(I)^{\alpha}$ pointwise
invariant. Hence $\beta$\ leaves $M_{loc}(A)^{\alpha}$ pointwise invariant,
therefore $\beta$\ commutes with $H.$\ Since $H$\ acts topologically
transitively on $M_{loc}(A),$\ applying Theorem 10, the conclusion follows.

g)$\Rightarrow$b) Suppose that g) holds. If $(A^{\alpha})^{\prime}\cap
M_{loc}(A)\neq%
\mathbb{C}
I,$ let $a\in(A^{\alpha})^{\prime}\cap M_{loc}(A)$ be a non scalar, positive
element of norm one. Then a standard argument shows that there exists a
unitary operator $v\in(A^{\alpha})^{\prime}\cap M_{loc}(A)$ such that
$a=\frac{1}{2}(v+v^{\ast})$ ($v=a+i(1-a^{2})^{\frac{1}{2}}).$ Since
$P^{\alpha}$\ is faithful we have $P^{\alpha}(a)\neq0,$ hence $P^{\alpha
}(v)\neq0.$ Now, $\beta=Adv$ is an automorphism of $M_{loc}(A)$\ that leaves
$A^{\alpha}$\ pointwise invariant. Therefore, there exists $g_{0}\in G$\ such
that $\beta=\alpha_{g_{0}}.$ Clearly $\alpha_{g_{0}}(v)=v.$ I will prove next
that $\alpha_{g}(v)=v$ for all $g\in G.$ We have for $a\in A$%
\[
\alpha_{g}(\alpha_{g_{0}}(a))=\alpha_{g}(vav^{\ast})=\alpha_{g}(v)\alpha
_{g}(a)\alpha_{g}(v^{\ast}).
\]
Since $G$ is abelian%
\[
\alpha_{g}(\alpha_{g_{0}}(a))=\alpha_{g_{0}}(\alpha_{g}(a))=v\alpha
_{g}(a)v^{\ast}.
\]
Therefore, $v^{\ast}\alpha_{g}(v)\in A^{\prime}\cap M_{loc}(A).$ Since $A$\ is
prime, by Lemma 3 b) it follows that $v^{\ast}\alpha_{g}(v)=\gamma(g)I$ for
some scalar $\gamma(g)$ with $\left\vert \gamma(g)\right\vert =1$ because
$v\ $is unitary Thus, $\alpha_{g}(v)=\gamma(g)v.$ It is obvious that
$\gamma(g)$ is a character of $G.$ By assumption, $P^{\alpha}(v)\neq0$ and so
$\sum_{g\in G}\gamma(g)\neq0$. This can only happen if $\gamma(g)=1$ for all
$g.$ Therefore, $\alpha_{g}(v)=v$ for all $g\in G.$ This means that
$v\in(A^{\alpha})^{\prime}\cap M_{loc}(A)^{\alpha}.$ By Lemma 3 a),
$M_{loc}(A)^{\alpha}\subset M_{loc}(A^{\alpha})$ so $v\in(A^{\alpha})^{\prime
}\cap M_{loc}(A^{\alpha}).$ Since $A^{\alpha}$\ is prime, applying Lemma 3 b)
to $A^{\alpha}$\ instead of $A$\ it follows that $v$ is a scalar so $a$\ is a
scalar. This contradiction proves that b) follows from g).
\end{proof}

\bigskip

\bigskip

\begin{center}
\bigskip
\end{center}

\bigskip

\textbf{1} \ \ \ \ \ P. Ara and M.Mathieu, \textit{Local Multipliers of
C*-Algebras, }Springer, 2003.

\textbf{2} \ \ \ \ \ O. Bratteli, G. A. Elliott, D. E. Evans and A. Kishimoto,
\textit{Quasi-Product Actions of a Compact Abelian Group on a C*-Algebra,
}Tohoku J. of Math., 41 (1989), 133-161.

\textbf{3} \ \ \ \ \ O. Bratteli, G. A. Elliott and D. W. Robinson,
\textit{Strong topological transitivity and C* -dynamical systems}, J. Math.
Soc. Japan Vol. 37, No. 1, 1985, 115-133.

\textbf{4} \ \ \ \ \ R. Dumitru, C. Peligrad, B. Visinescu,
\textit{Automorphisms inner in the local multiplier algebras and Connes
spectrum}, Operator Theory 20, Theta 2006, 75-80.

\textbf{5} \ \ \ \ \ A. Kishimoto, \textit{Outer Automorphisms and Reduced
Crossed Products of Simple C*-Algebras, }Commun. Math. Phys. 81 (1981), 429-435.

\textbf{6} \ \ \ \ \ A. Kishimoto, \textit{Freely acting automorphisms of
C*-algebras, }Yokohama Math. J. 30 (1982), 40-47.

\textbf{7} \ \ \ \ \ R. Longo and C. Peligrad, \textit{Noncommutative
topological dynamics and compact actions on C*-algebras, }J. Funct.
Analysis58, Nr.2 (1984), 157-174.

\textbf{8} \ \ \ \ \ D. Olesen and G. K. Pedersen, \textit{Applications of the
Connes Spectrum to C*-Dynamical Systems, III, }J. Funct. Analysis 45 (1982), 357-390.

\textbf{9} \ \ \ \ \ G. K. Pedersen, \textit{Approximating Derivations on
Ideals of C*-Algebras}, Inventiones Math. 45 (1978), 299-305.

\textbf{10} \ \ \ G. K. Pedersen, \textit{C*-Algebras and Their Automorphism
Groups, }Second edition, Academic Press, 2018.

\textbf{11} \ \ \ C. Peligrad, \textit{Properly outer and strictly outer
actions of finite groups on prime C*-algebras, }International J. of
Math\textit{., }Vol 35, No. 4 (2024), 11 pages\textit{.}

\textbf{12} \ \ \ C. Peligrad, \textit{Duality for Compact Group Actions on
Operator Algebras and Applications: Irreducible Inclusions and Galois
Correspondence,} Internat. J. Math. Vol 31, No. 9 (2020), 12 pp.

\textbf{13} \ \ \ S. Vaes, \textit{The Unitary Implementation of a Locally
Compact Group Action, }J. Functional Anal. 180 (2001), 426-480.

\textbf{14} \ \ \ N. E. Wegge-Olsen, \textit{K-Theory and C*-algebras,} Oxford
University Press, 1993.

\end{document}